\documentclass[11pt,dvipdfmx,a4paper]{article}
\usepackage[margin=30mm]{geometry}
\usepackage[colorlinks=true,linkcolor=blue,citecolor=blue]{hyperref}
\usepackage{amssymb,amsthm,amsmath, amsfonts,ascmac, enumerate}
\numberwithin{equation}{section}

\theoremstyle{definition}
\newtheorem{thm}{Theorem}[section]
\newtheorem{defn}[thm]{Definition}
\newtheorem{lem}[thm]{Lemma}
\newtheorem{prop}[thm]{Proposition}
\newtheorem{rem}[thm]{Remark}

\newtheorem{ex}[thm]{Example}

\newcommand{\R}{\mathbb{R}}   
\newcommand{\C}{\mathbb{C}}

\newcommand{\calB}{{\mathcal{B}}}

\newcommand{\calF}{{\mathcal{F}}} 

\newcommand{\calH}{{\mathcal{H}}}

\newcommand{\calM}{{\mathcal{M}}}

\newcommand{\calP}{{\mathcal{P}}}
\newcommand{\calQ}{{\mathcal{Q}}}

\title{Monotone max-convolution and subordination functions for free max-convolution}

\author{Yuki Ueda}
\date{\today}

\begin{document}

\maketitle


\begin{abstract}
We show that the distribution of the spectral maximum of monotonically independent self-adjoint operators coincides with the classical max-convolution of their distributions. In free probability, it was proven that for any probability measures $\sigma,\mu$ on $\mathbb{R}$ there is a unique probability measure $\mathbb{A}_\sigma(\mu)$ satisfying $\sigma\boxplus \mu = \sigma \triangleright \mathbb{A}_\sigma(\mu)$, where $\boxplus$ and $\triangleright$ are free and monotone additive convolutions, respectively. We recall that the reciprocal Cauchy transform of $\mathbb{A}_\sigma(\mu)$ is the subordination function for free additive convolution. Motivated by this analogy, we introduce subordination functions for free max-convolution and prove their existence and structural properties. 
\end{abstract}

%%%

\section{Introduction}

In classical probability theory, not only the addition of independent random variables and the additive convolution play a central role, but so do the notions of the maximum of independent real-valued random variables and the associated max-convolution. We consider independent real-valued random variables $X,Y$ on a probability space $(\Omega,\calF,\mathbb{P})$. For any $x\in \R$, one can see that
\begin{align*}
\mathbb{P}(\max\{X,Y\} \le x ) =\mathbb{P}(X \le x, Y \le x) =\mathbb{P}(X\le x) \mathbb{P}(Y\le x).
\end{align*}
Let us set $\calP(I)$ as the class of Borel probability measures on $I\subset \R$. For $\mu,\nu \in \calP(\R)$, the {\it max-convolution} $\mu\lor \nu$ is defined as a probability measure whose distribution function is given by
$$
F_{\mu\lor \nu} (x) := F_\mu(x) F_\nu(x), \qquad x\in\R,
$$
where $F_\mu$ is a distribution function of $\mu$. 
The theory base on max-convolution has been developed both for studying the statistics of maxima in data sets and for advancing the underlying mathematical framework. One of the most important results in max-convolution is the limit theorem by Fisher--Tippett--Gnedenko \cite{FT28, G43}. For further details on max-probability theory (or extreme value theory), the reader may consult \cite{R} and the references therein.

By analogy, Ben Arous and Voiculescu \cite{BV06} introduced max-convolution and extreme value distributions in the framework of free probability theory (see \cite{MS17, NS06} and the references therein for foundational material on free probability). Let $(\calM,\varphi)$ be a $W^\ast$-probability space, that is, $\calM$ is a von Neumann algebra and $\varphi$ is a normal faithful tracial state. We denote by $\calM_{sa}$ the set of all self-adjoint operators in $\calM$. We may assume that $\calM$ acts on some Hilbert space $\calH$, that is, $\calM$ can be regarded as a $\ast$-subalgebra of $B(\calH)$ the algebra of all bounded linear operators on $\calH$. In non-commutative probability theory, unbounded operators can be treated as follows. We define $\overline{\calM}_{sa}$ as the set of all self-adjoint operators $a$ affiliated with $\calM$, that is, for any bounded Borel measurable function $f$ on $\R$, the functional calculus $f(a)$ belongs to $\calM$. It is known that for any $a,b \in \overline{\calM}_{sa}$, the closure of $a+b$ belongs to $\overline{\calM}_{sa}$. If $a\in \overline{\calM}_{sa}$, then the function
$$
x\mapsto \varphi(E_a((-\infty, x]))
$$
is a distribution function, where $E_a(B)$ denotes the spectral projection of $a$ onto a Borel set $B\subset \R$. Note that $E_a(B)= \mathbf{1}_B(a) \in \calM$ for all Borel sets $B$. There exists a unique measure $\mu\in \calP(\R)$ such that 
$$
\varphi(E_a((-\infty, x])) = F_\mu(x).
$$
The measure $\mu$ is called the {\it spectral distribution of $a$ with respect to $\varphi$}. 

In the following, we explain the concept of the maximum of non-commutative random variables. A partial order $\prec$ on the set $\calM_{sa}$, as well as the notion of the maximum of self-adjoint operators with respect to $\prec$, was introduced by \cite{A89, O71}. More precisely, for $a,b \in \calM_{sa}$, we define
$$
a\prec b \Leftrightarrow E_a((x,\infty)) \le E_b((x,\infty)) \ \text{for all } x\in \R,
$$
where $P \le Q$ denotes the usual operator order, that is, $Q-P\ge0$. For $a,b \in \calM_{sa}$, we define $a\lor b$ as the self-adjoint operator in $\calM$ whose spectral projection is given by
\begin{align}\label{eq:spectral_maximum}
E_{a\lor b}((-\infty, x]):= E_a((-\infty,x]) \land E_b((-\infty,x]), \qquad x \in \R,
\end{align}
where $P\land Q$ is the minimum of projections $P$ and $Q$ with respect to the usual operator order. By definition, we get $a,b \prec a\lor b$. The operator $a\lor b$ is called the {\it spectral maximum} of $a$ and $b$. Subsequently, Ben Arous and Voiculescu \cite{BV06} extended the definition of $\prec$ to the class $\overline{\calM}_{sa}$ and studied the spectral maximum of freely independent operators in $\overline{\calM}_{sa}$ as follows.

\begin{thm}[Ben Arous and Voiculescu, \cite{BV06}]
Let $(\calM,\varphi)$ be a $W^\ast$-probability space and $X,Y$ freely independent self-adjoint operators affiliated with $\calM$. Then we have
$$
\varphi(E_{X\lor Y}((-\infty, x])) = \max\{\varphi(E_X((-\infty, x]))+ \varphi(E_Y((-\infty, x]))-1,0 \}, \quad x\in \R.
$$
\end{thm}
Motivated by the above result, for $\mu,\nu\in \calP(\R)$, we define the {\it free max-convolution} $\mu \Box\hspace{-.89em}\lor \nu$ as the probability measure whose distribution function is given by
$$
F_{\mu \Box\hspace{-.55em}\lor \nu} (x) := \max\{F_\mu(x) + F_\nu(x)-1,0\}, \qquad x\in \R.
$$
The free max-convolution power of $\mu\in\calP(\R)$ is defined by
$$
F_{\mu^{\Box\hspace{-.47em}\lor t}}(x) := \max\{tF_\mu(x)-(t-1),0\}, \qquad x\in \R, \quad t\ge 1,
$$
see \cite{U21} for details of free max-convolution power.

In 2021, Vargas and Voiculescu \cite{VV21} introduced the Boolean max-convolution and studied extreme value theory in the context of Boolean independence (see \cite{SW97} for a detailed introduction to Boolean probability theory). Due to certain constraints in the Boolean framework and in operator theory, it is necessary to restrict attention to the set of positive operators on a Hilbert space in order to define the (non-trivial) spectral maximum of Boolean independent random variables.

\begin{thm}[Vargas and Voiculescu, \cite{VV21}]
Let $\calH$ be a Hilbert space and $\xi\in \calH$ a unit vector. We define the vector state
$$
\varphi_\xi(X):= \langle X\xi, \xi \rangle_{\calH}, \qquad X\in B(\calH).
$$
If $X,Y$ are Boolean independent positive operators on $\calH$, then
\begin{align*}
\varphi_\xi& (E_{X\lor Y}([0,x])) \\
&= \frac{\varphi_\xi (E_{X}([0,x])) \varphi_\xi (E_{Y}([0,x]))}{\varphi_\xi (E_{X}([0,x]))+ \varphi_\xi (E_{Y}([0,x]))-\varphi_\xi (E_{X}([0,x])) \varphi_\xi (E_{Y}([0,x]))}, \quad x\ge0.
\end{align*}
\end{thm}
For $\mu,\nu \in \calP(\R_{\ge0})$, the {\it Boolean max-convolution} $\mu\cup \hspace{-.89em}\lor \nu$ is defined as the probability measure on $\R_{\ge0}$ whose distribution function is given by
$$
F_{\mu\cup \hspace{-.52em}\lor \nu}(x) := \frac{F_{\mu}(x)F_{\nu}(x)}{F_{\mu}(x)+F_{\nu}(x)-F_{\mu}(x)F_{\nu}(x)}, \quad x\ge 0,
$$
where $F_{\mu\cup \hspace{-.52em}\lor \nu}(x)=0$ if either $F_\mu(x)=0$ or $F_\nu(x)=0$. The Boolean max-convolution power of $\mu \in \calP(\R_{\ge0})$ is defined by
$$
F_{\mu^{\cup \hspace{-.47em}\lor t}}(x) = \frac{F_\mu(x)}{t-(t-1)F_\mu(x)}, \quad x\ge 0, \quad t>0,
$$
see \cite{U21} for details of Boolean max-convolution power.

In this paper, we develop the concept of max-convolution in the context of monotone independence for non-commutative random variables. Monotone independence was introduced by Muraki \cite{M00,M01} as a form of universal independence distinct from tensor, free and Boolean independences. In Section \ref{sec2}, we compute the spectral distribution of the spectral maximum of monotonically independent random variables.

\begin{thm}\label{main1}
Let $(\calH,\xi)$ be a pair of a Hilbert space $\calH$ and a unit vector $\xi \in \calH$. Consider monotonically independent self-adjoint operators $X,Y$ on $\calH$, distributed as $\mu,\nu$ with respect to $\varphi_\xi$, respectively. Then we obtain
$$
\varphi_\xi (E_{X\lor Y}((-\infty,x]))= F_{\mu\lor \nu}(x), \quad x\in\R.
$$
\end{thm}

To analytically study the {\it free additive convolution} $\boxplus$ which describes the distribution of the sum of freely independent random variables (see \cite{BV93} for details), Biane \cite{B98}, Belinschi and Bercovici \cite{BB07} developed subordination functions for $\boxplus$. For $\mu \in \calP(\R)$, the Cauchy transform $G_\mu$ and the reciprocal Cauchy transform $H_\mu$ are defined by
$$
G_\mu(z) :=\int_{\R}\frac{\mu({\rm d}u)}{z-u} \quad \text{and} \quad H_\mu(z):=\frac{1}{G_\mu(z)}, \qquad z\in \C^+.
$$
In particular, if a non-commutative random variable $X$ has a distribution $\mu\in\calP(\R)$, we sometimes write $G_X(z)$ and $H_X(z)$ in place of $G_\mu(z)$ and $H_\mu(z)$, respectively. According to \cite[Theorem 4.1]{BB07}, for $\sigma,\mu \in \calP(\R)$, there exist analytic functions $w_1, w_2: \C^+\to \C^+$ such that $$
\lim_{y\to \infty} \frac{w_i(iy)}{iy}=1, \quad i=1,2,
$$
and
$$
H_{\sigma\boxplus \mu} (z)= H_\sigma(w_1(z))=H_\mu(w_2(z)),\qquad z\in \C^+,
$$
and also
$$
w_1(z)+w_2(z)=z+H_{\sigma\boxplus\mu}(z), \quad z\in \C^+.
$$
The functions $w_1,w_2$ are said {\it subordination functions for free additive convolution}. A characterization of reciprocal Cauchy transforms (see \cite[Proposition 5.2]{BV93}) implies that there exists a unique measure $\mathbb{A}_\sigma(\mu)\in\calP(\R)$ such that $w_1(z)= H_{\mathbb{A}_\sigma(\mu)}(z)$ (similarly, $w_2(z)=H_{\mathbb{A}_\mu(\sigma)}(z)$). Consequently, we have
$$
\sigma\boxplus \mu = \sigma \triangleright \mathbb{A}_\sigma(\mu),
$$
where $\triangleright$ denotes the {\it monotone additive convolution}, which gives the distribution of the sum of monotonically independent random variables and is characterized by $H_{\rho_1\triangleright \rho_2}(z)=H_{\rho_1}(H_{\rho_2}(z))$ for all $\rho_1,\rho_2\in\calP(\R)$ and $z\in \C^+$ (see \cite{M00,M01} for details).  According to \cite[Theorem 1.3]{AH16}, for any $\sigma_1,\sigma_2,\sigma,\mu_1,\mu_2\in \calP(\R)$, we have
\begin{enumerate}[\rm (i)]
\item $\mathbb{A}_{\sigma_1}\circ \mathbb{A}_{\sigma_2}=\mathbb{A}_{\sigma_1 \triangleright \sigma_2}$ on $\calP(\R)$.
\item $\mathbb{A}_\sigma(\mu_1\boxplus \mu_2) =\mathbb{A}_\sigma(\mu_1) \boxplus \mathbb{A}_\sigma(\mu_2)$.
\end{enumerate}
The following formula was not explicitly stated in \cite{AH16}, but it is easily verified that $\mathbb{A}_\sigma(\mu^{\boxplus t}) = \mathbb{A}_\sigma(\mu)^{\boxplus t}$ for any $\mu\in \calP(\R)$ and $t\ge 1$, where $\rho^{\boxplus t}$ denotes the free additive convolution power of $\rho\in \calP(\R)$ (see \cite{BB04}). Furthermore, according to \cite[Theorem 4.1]{BB07}, for any $\sigma, \mu\in \calP(\R)$, their subordination functions satisfy the following relation:
$$
H_{\sigma\boxplus \mu}(z) = H_{\mathbb{A}_\sigma(\mu)}(z) +H_{\mathbb{A}_\mu(\sigma)}(z) -z, \qquad z\in \C^+.
$$
This relation implies the formula $\sigma\boxplus \mu = \mathbb{A}_\sigma(\mu) \uplus \mathbb{A}_\mu(\sigma)$, where $\uplus$ denotes the {\it Boolean additive convolution} (see \cite{SW97}). In Section \ref{sec3}, we present the max-analogue of the results discussed above. First, we prove that, for $\sigma,\mu\in \calP(\R)$, there exists a unique measure $\mathbb{A}^\lor_\sigma(\mu)\in \calP(\R)$ such that
$$
\sigma\Box\hspace{-.92em}\lor \mu = \sigma \lor \mathbb{A}^\lor_\sigma(\mu).
$$
This is the max-analogue of the subordination relation in free probability. In contrast to the subordination function in free probability, which is described in terms of analytic transforms, the max-version admits a simple and explicit expression at the level of distribution functions. More precisely, the distribution function of $\mathbb{A}_\sigma^\lor (\mu)$ is given by
$$
F_{\mathbb{A}^\lor_\sigma(\mu)}(x) = \max\left\{ 1- \frac{\overline{F}_\mu(x)}{F_\sigma(x)},0\right\},
$$
where $\overline{F}_\mu:=1-F_\mu$ denotes the tail distribution of $\mu$. We call $F_{\mathbb{A}^\lor_\sigma(\mu)}$ the {\it subordination function for free max-convolution}. Finally, we establish the following theorem.
\begin{thm}\label{main2}
For any $\sigma_1, \sigma_2, \sigma, \mu_1,\mu_2, \mu \in \calP(\R)$ and $t\ge 1$, we have
\begin{enumerate}[\rm (1)]
\item $\mathbb{A}^\lor_{\sigma_1} \circ \mathbb{A}^\lor_{\sigma_2} = \mathbb{A}^\lor_{\sigma_1 \lor \sigma_2}$ on $\calP(\R)$.
\item $\mathbb{A}^\lor_{\sigma}(\mu_1  \Box\hspace{-.92em} \lor \mu_2)=\mathbb{A}^\lor_{\sigma}(\mu_1)\Box\hspace{-.92em} \lor \mathbb{A}^\lor_{\sigma}  (\mu_2)$.
\item $\mathbb{A}^\lor_{\sigma}(\mu^{\Box\hspace{-.55em}\lor t}) = \mathbb{A}^\lor_{\sigma}(\mu)^{\Box\hspace{-.55em}\lor t}$. 
\item If $\sigma,\mu \in \calP(\R_{\ge 0})$, then $\sigma \Box \hspace{-.92em}\lor \mu = \mathbb{A}^\lor_\sigma(\mu) \cup \hspace{-.92em} \lor \ \mathbb{A}^\lor_\mu(\sigma)$.
\end{enumerate}
\end{thm}

Next, we examine the Boolean max-analogue of the preceding discussion. However, it is necessary to restrict the conditions for considering probability measures. For a measure $\mu \in \calP(\R_{\ge0})$, we define
$$
\alpha_\mu:=\sup\{x: F_\mu(x)=0\} \ge0.
$$
Note that, $\alpha_\mu=0$ if and only if $F_\mu(x)>0$ for all $x>0$. Furthermore, we put
$$
\calQ:= \left\{(\sigma,\mu)\in \calP(\R_{\ge0})^2:  \frac{F_\sigma}{F_\mu} \text{ is monotonically decreasing on } (\alpha_\mu,\infty) \text{ and } \mu(\{\alpha_\mu\})=0\right\}.
$$
We show that, for any $(\sigma,\mu) \in \calQ$, there exists a unique measure $\mathbb{U}_\sigma^\lor (\mu)\in \calP(\R_{\ge0})$ such that
$$
\sigma \cup \hspace{-.88em}\lor \mu = \sigma \lor \mathbb{U}_\sigma^\lor (\mu).
$$
Moreover, we prove the following analogous result.
\begin{thm}\label{main3}
The following formulas hold:
\begin{enumerate}[\rm (1)]
\item If $(\sigma_1\lor \sigma_2, \mu) \in \calQ$, then $(\sigma_i, \mu)\in\calQ$ for all $i=1,2$. In addition, if $(\sigma_1, \mathbb{U}^\lor_{\sigma_2}(\mu))\in \calQ$, then 
$(\mathbb{U}^\lor_{\sigma_1} \circ \mathbb{U}^\lor_{\sigma_2} )(\mu) = \mathbb{U}^\lor_{\sigma_1 \lor \sigma_2} (\mu)$.
\item If $(\sigma,\mu_1), (\sigma,\mu_2)\in \calQ$, then $\mathbb{U}^\lor_{\sigma}(\mu_1  \cup\hspace{-.88em} \lor \mu_2)=\mathbb{U}^\lor_{\sigma}(\mu_1)\cup\hspace{-.88em} \lor \ \mathbb{U}^\lor_{\sigma}  (\mu_2)$.
\item If $(\sigma,\mu)\in \calQ$ and $t\ge 1$, then $\mathbb{U}^\lor_{\sigma}(\mu^{\cup\hspace{-.52em}\lor t}) = \mathbb{U}^\lor_{\sigma}(\mu)^{\cup\hspace{-.52em}\lor t}$.
\item For any $\sigma \in \calP(\R_{\ge0})$ with $\sigma(\{\alpha_\sigma\})=0$, we have $\sigma^{\cup\hspace{-.52em}\lor 2} =\mathbb{U}^\lor_\sigma(\sigma)^{\Box \hspace{-.52em}\lor 2}$.
\end{enumerate}
\end{thm}

This paper is organized as follows. In Section \ref{sec2}, we study the spectral maximum of monotonically independent random variables and prove Theorem \ref{main1}. In Section \ref{sec3}, we construct subordination functions for free max-convolution. Finally, we establish a proof of Theorem \ref{main2}. In Section \ref{sec4}, we establish the Boolean analogue of Section \ref{sec3} and finally prove Theorem \ref{main3}.

%%%
\section{Monotone max-convolution}\label{sec2}

We study a concrete realization of the spectral maximum of monotonically independent projections on a Hilbert space. Let $(\calH_1,\xi_1)$ and $(\calH_2,\xi_2)$ be pairs of a Hilbert space and a unit vector. Consider $X\in B(\calH_1)_{sa}$ and $Y\in \calB(\calH_2)_{sa}$. We define $\calH:=\calH_1\otimes \calH_2$, $\xi = \xi_1 \otimes \xi_2 \in \calH$ and
$$
\widetilde{X}:= X\otimes P_{\C\xi_2}, \qquad \widetilde{Y}:= I_{\calH_1}\otimes Y,
$$
where $P_{\C\xi_2}$ is the orthogonal projection onto $\C\xi_2$. Then $\widetilde{X}, \widetilde{Y} \in B(\calH)_{sa}$ are monotonically independent with respect to $\varphi_{\xi}$. Note that
\begin{align*}
\varphi_\xi(E_{\widetilde{X}}((-\infty,t]) ) &=\varphi_{\xi_1}(E_X((-\infty,t])),\\
\varphi_\xi(E_{\widetilde{Y}}((-\infty,t])) &=\varphi_{\xi_2}(E_Y((-\infty,t]))
\end{align*}
for all $t\in \R$. Moreover,
\begin{align*}
E_{\widetilde{X}}((-\infty,t]) = \begin{cases}
E_X((-\infty,t])\otimes P_{\C\xi_2}, & t<0,\\
E_X((-\infty,t])\otimes P_{\C\xi_2} + I_{\calH_1} \otimes P_{\calH_2\ominus \C\xi_2}, & t\ge 0,
\end{cases}
\end{align*}
and 
$$
E_{\widetilde{Y}}((-\infty,t])  = I_{\calH_1} \otimes E_Y((-\infty,t])), \qquad t\in \R.
$$
The spectral maximum of $\widetilde{X}$ and $\widetilde{Y}$ is given by the spectral projections 
$$
E_{\widetilde{X}\lor \widetilde{Y}}((-\infty,t]) = E_{\widetilde{X}}((-\infty,t]) \land  E_{\widetilde{Y}}((-\infty,t])
$$ 
for all $t\in \R$, and $\langle E_{\widetilde{X}\lor \widetilde{Y}}((-\infty,t]) \xi, \xi\rangle \in [0,1]$ for $t\in \R$.

\begin{prop}\label{prop1}
Let $(\calH,\xi)$ be a pair of a Hilbert space and a unit vector. If $P,Q$ are monotonically independent projections on $\calH$ such that $\varphi_\xi(P)=1-p$ and $\varphi_\xi(Q)=1-q$ for some $p,q\in [0,1]$, then $\varphi_\xi(P\lor Q)=1-pq$.
\end{prop}
\begin{proof}
Notice that, $E_{P\land Q}(\{0\})= E_{P+Q}(\{0\})$. Thus, if we write $\varphi_{\xi}(P\lor Q)=1-r$, then
$$
r = \lim_{z\to 0} zG_{P+Q}(z).
$$
Since $P,Q$ are monotonically independent, we get
\begin{align*}
H_{P+Q}(z)
& = H_P(H_Q(z)) = \frac{H_Q(z)(H_Q(z)-1)}{H_Q(z)-p}\\
&= \frac{\frac{z(z-1)}{z-q} \left(\frac{z(z-1)}{z-q}-1\right)}{\frac{z(z-1)}{z-q}-p} = \frac{z(z-1)(z(z-1)-(z-q))}{z(z-1)(z-q)-p(z-q)^2}.
\end{align*}
Therefore
$$
r=\lim_{z\to 0} zG_{P+Q}(z) = \lim_{z\to 0} \frac{z(z-1)(z-q)-p(z-q)^2}{(z-1)(z(z-1)-(z-q))} =pq.
$$
Hence $\varphi_\xi(P\lor Q)=1-pq$.
\end{proof}

As a consequence of Proposition \ref{prop1}, we can compute the spectral distribution of the spectral maximum of monotonically independent self-adjoint operators.

\begin{proof}[Proof of Theorem \ref{main1}]
By Proposition \ref{prop1}, we have
\begin{align*}
\varphi_\xi (E_{X\lor Y}((-\infty, x])) 
&= \varphi_\xi (E_X((-\infty, x]) \land E_Y((-\infty, x]))\\
&=\varphi_\xi (I_{\calH} - (E_X((x,\infty)) \lor E_Y((x,\infty))))\\
&=1- \varphi_\xi(E_X((x,\infty)) \lor E_Y((x,\infty)))\\
&= 1- \left\{1- \varphi_\xi (E_X((-\infty,x])) \varphi_\xi (E_Y((-\infty,x]))\right\}\\
&=\varphi_\xi (E_X((-\infty,x])) \varphi_\xi (E_Y((-\infty,x]))\\
&=F_\mu(x) F_\nu(x)\\
&=F_{\mu\lor \nu}(x),
\end{align*}
as desired.
\end{proof}

\begin{rem}
We can consider the spectral distribution of $X\lor Y$ even when $X$ and $Y$ are unbounded self-adjoint operators via the spectral construction \eqref{eq:spectral_maximum}. For any probability measure $\mu,\nu\in \calP(\R)$, the {\it monotone max-convolution} $\mu \triangledown \nu$ is defined as the probability measure on $\R$ that coincides with the distribution $\mu\lor \nu$. Naturally, the class of monotone extreme value distributions coincides with the class of classical extreme value distributions. 

In general, by definition of monotone additive convolution, we have $\mu \triangleright \nu \neq \nu \triangleright \mu$ for any $\mu,\nu\in\calP(\R)$. However, as the above discussion shows, the monotone max-convolution is commutative: $\mu \triangledown \nu = \nu \triangledown  \mu$ for any $\mu,\nu\in\calP(\R)$.
\end{rem}

%%%
\section{Subordination functions for free max-convolution}
\label{sec3}

In this section, we construct the max-analogue of subordination functions for free additive convolution. For $\sigma,\mu \in \calP(\R)$, we define
$$
U_{\sigma,\mu}:=\{ x \in \R : \overline{F}_\mu(x) < F_\sigma(x)\}.
$$
Since $\overline{F}_\mu$ is a non-increasing function with $\lim_{x\to \infty}\overline{F}_\mu=0$ and $F_\sigma$ is a non-decreasing function with $\lim_{x\to \infty}\overline{F}_\sigma=1$, we have $U_{\sigma, \mu}  \neq \emptyset$.  Note that $U_{\sigma,\mu}=U_{\mu,\sigma}$.

\begin{prop}\label{prop2}
For any $\sigma,\mu \in \calP(\R)$, there exists a unique measure $\mathbb{A}^\lor_\sigma(\mu) \in \calP(\R)$ such that
$$
\sigma \Box\hspace{-.92em} \lor \mu = \sigma \lor \mathbb{A}^\lor_\sigma(\mu).
$$
\end{prop}
\begin{proof}
For $\sigma,\mu \in \calP(\R)$, we define
$$
F:= \max\left\{ 1- \frac{\overline{F}_\mu}{F_{\sigma}}, 0 \right\},
$$
where we understand $F(x)=0$ if $F_\sigma(x)=0$. It is straightforward to verify that $F=1- \frac{\overline{F}_\mu}{F_{\sigma}}$ on $U_{\sigma,\mu}$ and $F$ is right continuous. If $x_1< x_2$ in $U_{\sigma,\mu}$, then
$$
\frac{\overline{F}_\mu(x_1)}{F_{\sigma}(x_1)} \ge \frac{\overline{F}_\mu(x_2)}{F_{\sigma}(x_2)},
$$
and therefore $F(x_1) \le F(x_2)$. For sufficiently large $x$, we may assume $x\in U_{\sigma,\mu}$. Then $\lim_{x\to \infty}F(x)=1$. If $x$ is sufficiently large and negative, then $x\notin U_{\sigma,\mu}$, and hence $\lim_{x\to -\infty}F(x)=0$. Consequently, $F$ is a distribution function on $\R$. Hence, there exists a unique probability measure $\mathbb{A}_\sigma^\lor (\mu)$ on $\R$ such that $F=F_{\mathbb{A}_\sigma^\lor (\mu)}$. By definition, we have
\begin{align*}
F_{\sigma \lor \mathbb{A}_\sigma^\lor (\mu)} 
&= F_\sigma F_{\mathbb{A}_\sigma^\lor (\mu)}= F_\sigma \max\left\{1- \frac{\overline{F}_\mu}{F_{\sigma}}, 0 \right\}\\
&=\max\{F_\sigma- \overline{F}_\mu,0\} = \max\{F_\sigma+F_\mu-1,0\} = F_{\sigma\Box\hspace{-.48em}\lor \mu}.
\end{align*}
\end{proof}

\begin{defn}
For $\sigma,\mu\in \calP(\R)$, we define the probability measure $\mathbb{A}_\sigma^\lor(\mu)$ on $\R$ by
$$
F_{\mathbb{A}^\lor_\sigma(\mu)} := \max\left\{1-\frac{\overline{F}_\mu}{F_\sigma},0\right\},
$$
where $F_{\mathbb{A}^\lor_\sigma(\mu)} (x)=0$ if $F_\sigma(x)=0$.
We call the distribution function $F_{\mathbb{A}^\lor_\sigma(\mu)}$ the {\it subordination function for free max-convolution}. 
\end{defn}

Finally, we prove the main theorem as follows.

\begin{proof}[Proof of Theorem \ref{main2}]
Let us consider $\sigma_1, \sigma_2, \sigma, \mu_1,\mu_2, \mu \in \calP(\R)$ and $t\ge 1$.
\begin{enumerate}[\rm (1)]
\item For any $\mu \in \calP(\R)$, we have
\begin{align*}
F_{\mathbb{A}^\lor_{\sigma_1} \circ \mathbb{A}^\lor_{\sigma_2} (\mu)}
&=F_{\mathbb{A}^\lor_{\sigma_1} (\mathbb{A}^\lor_{\sigma_2} (\mu))}= \max \left\{1 - \frac{\overline{F}_{\mathbb{A}^\lor_{\sigma_2} (\mu)}}{F_{\sigma_1}}, 0 \right\}\\
&= \max\left\{ 1 - \frac{1}{F_{\sigma_1}} \left( 1- \max\left\{ 1- \frac{\overline{F}_\mu}{F_{\sigma_2}},0 \right\}\right),0 \right\}\\
&= \max \left\{1-\frac{1}{F_{\sigma_1}} + \max\left\{\frac{1}{F_{\sigma_1}}- \frac{\overline{F}_\mu}{F_{\sigma_1}F_{\sigma_2}},0 \right\} ,0 \right\}.
\end{align*}
On $U_{\sigma_2,\mu}$, we have
\begin{align*}
F_{\mathbb{A}^\lor_{\sigma_1} \circ \mathbb{A}^\lor_{\sigma_2} (\mu)}
&= \max\left\{1-\frac{1}{F_{\sigma_1}} + \frac{1}{F_{\sigma_1}} - \frac{\overline{F}_\mu}{F_{\sigma_1} F_{\sigma_2}},0 \right\}\\
&= \max\left\{ 1 - \frac{\overline{F}_\mu}{F_{\sigma_1 \lor \sigma_2}},0 \right\} = F_{\mathbb{A}^\lor_{\sigma_1\lor \sigma_2}(\mu)}.
\end{align*}
On $\R \setminus U_{\sigma_2,\mu}$, we obtain
$$
F_{\mathbb{A}^\lor_{\sigma_1} \circ \mathbb{A}^\lor_{\sigma_2} (\mu)} = \max\left\{1-\frac{1}{F_{\sigma_1}},0\right\} =0.
$$
On the other hand, since $1 - \frac{\overline{F}_\mu}{F_{\sigma_1} F_{\sigma_2}} \le 1 -\frac{1}{F_{\sigma_1}}<0$, we have
$F_{\mathbb{A}^\lor_{\sigma_1\lor \sigma_2}(\mu)} = 0$.

\item A straightforward calculation shows that
$$
F_{\mathbb{A}^\lor_{\sigma}(\mu_1)\Box\hspace{-.52em} \lor \mathbb{A}^\lor_{\sigma}  (\mu_2)} =\max\left\{\max\left\{ 1- \frac{\overline{F}_{\mu_1}}{F_{\sigma}},0 \right\}+ \max\left\{ 1- \frac{\overline{F}_{\mu_2}}{F_{\sigma}},0 \right\}-1,\ 0 \right\}.
$$
It is easy to verify that if $x\in U_{\sigma,\mu_1} \cap U_{\sigma, \mu_2}$ then
\begin{align*}
F_{\mathbb{A}^\lor_{\sigma}(\mu_1)\Box\hspace{-.52em} \lor \mathbb{A}^\lor_{\sigma}  (\mu_2)} (x) 
&= \max\left\{ 1- \frac{\overline{F}_{\mu_1}(x)+ \overline{F}_{\mu_2}(x)}{F_{\sigma}(x)}, 0 \right\}\\
&=\begin{cases}
1- \frac{\overline{F}_{\mu_1}(x)+ \overline{F}_{\mu_2}(x)}{F_{\sigma}(x)}, &  x\in \{\overline{F}_{\mu_1} + \overline{F}_{\mu_2}< F_\sigma\}\\
0, & x\in \{ F_\sigma \le \overline{F}_{\mu_1}+ \overline{F}_{\mu_2}\}.
\end{cases}
\end{align*}
If $x\notin U_{\sigma,\mu_1} \cap U_{\sigma, \mu_2}$, then $F_{\mathbb{A}^\lor_{\sigma}(\mu_1)\Box\hspace{-.52em} \lor \mathbb{A}^\lor_{\sigma}  (\mu_2)} (x) =0$.

On the other hand, since $\overline{F}_{\mu_1\Box\hspace{-.52em}\lor \mu_2} = \min\{\overline{F}_{\mu_1}+ \overline{F}_{\mu_2},1\}$, we have
$$
F_{\mathbb{A}^\lor_{\sigma}(\mu_1  \Box\hspace{-.52em} \lor \mu_2)} = \max \left\{ 1 - \frac{\min\{\overline{F}_{\mu_1}+ \overline{F}_{\mu_2},1\}}{F_\sigma},0 \right\}.
$$
If $x\notin U_{\sigma,\mu_1} \cap U_{\sigma, \mu_2}$, then
$$
1- \frac{\overline{F}_{\mu_1}+ \overline{F}_{\mu_2}}{F_\sigma}<-1,
$$
and therefore $F_{\mathbb{A}^\lor_{\sigma}(\mu_1  \Box\hspace{-.52em} \lor \mu_2)}(x)=0$. Next, we suppose that $x\in U_{\sigma,\mu_1} \cap U_{\sigma, \mu_2}$. Furthermore, if $ \overline{F}_{\mu_1}(x) + \overline{F}_{\mu_2}(x) < F_\sigma(x)$, then $\min\{\overline{F}_{\mu_1}(x)+ \overline{F}_{\mu_2}(x),1\}=\overline{F}_{\mu_1}(x)+ \overline{F}_{\mu_2}(x)$, and therefore
$$
F_{\mathbb{A}^\lor_{\sigma}(\mu_1  \Box\hspace{-.52em} \lor \mu_2)}(x) = 1- \frac{\overline{F}_{\mu_1}(x)+ \overline{F}_{\mu_2}(x)}{F_{\sigma}(x)}.
$$
If $F_\sigma(x) \le \overline{F}_{\mu_1}(x) + \overline{F}_{\mu_2}(x)$, then
$$
\frac{\min \{ \overline{F}_{\mu_1}(x) +\overline{F}_{\mu_2}(x), 1\} }{F_\sigma(x)}  \ge 1,
$$
and therefore $F_{\mathbb{A}^\lor_{\sigma}(\mu_1  \Box\hspace{-.52em} \lor \mu_2)}(x)=0$. Finally, we get $F_{\mathbb{A}^\lor_{\sigma}(\mu_1  \Box\hspace{-.52em} \lor \mu_2)} = F_{\mathbb{A}^\lor_{\sigma}(\mu_1)\Box\hspace{-.52em} \lor \mathbb{A}^\lor_{\sigma}  (\mu_2)}$.

\item First, we have
\begin{align*}
F_{\mathbb{A}^\lor_\sigma(\mu)^{\Box\hspace{-.47em}\lor t}}= \max \left\{ t \max\left\{1-\frac{\overline{F}_\mu}{F_\sigma}, 0 \right\}- (t-1), 0 \right\}.
\end{align*}
For $x\in U_{\sigma,\mu}$, we get 
\begin{align*}
F_{\mathbb{A}^\lor_\sigma(\mu)^{\Box\hspace{-.47em}\lor t}} (x)
&= \max\left\{1-\frac{t \overline{F}_\mu(x)}{F_\sigma(x)},0\right\}\\
&=\begin{cases}
1-\frac{t\overline{F}_\mu(x)}{F_\sigma(x)}, & x\in \{\overline{F}_\mu < \frac{1}{t} F_{\sigma}\}\\
0, & x\in \{\frac{1}{t} F_{\sigma} \le \overline{F}_\mu < F_{\sigma}\}.
\end{cases}
\end{align*}
If $x\notin U_{\sigma,\mu}$, it is easy to see that $F_{\mathbb{A}^\lor_\sigma(\mu)^{\Box\hspace{-.47em}\lor t}} (x)=0$.

On the other hand, we have
\begin{align*}
F_{\mathbb{A}^\lor_{\sigma}(\mu^{\Box\hspace{-.47em}\lor t})}= \max\left\{1-\frac{1}{F_\sigma} + \frac{1}{F_\sigma} \max\{tF_\mu-(t-1),0\},0 \right\}.
\end{align*}
Clearly, if $x\notin U_{\sigma,\mu}$, then $F_{\mathbb{A}^\lor_{\sigma}(\mu^{\Box\hspace{-.47em}\lor t})} (x)=0$. Suppose that $x\in U_{\sigma,\mu}$. Moreover if $\overline{F}_\mu(x) < \frac{1}{t}F_\sigma(x)$, then $tF_\mu- (t-1) > 1-F_\sigma$. Therefore
\begin{align*}
F_{\mathbb{A}^\lor_{\sigma}(\mu^{\Box\hspace{-.47em}\lor t})}(x) = \max\left\{1 -\frac{t\overline{F}_\mu(x)}{F_\sigma(x)},0 \right\} = 1-\frac{t\overline{F}_\mu(x)}{F_\sigma(x)}.
\end{align*}
If $\frac{1}{t}F_\sigma(x) \le \overline{F}_\mu < F_\sigma$, then 
$$
1-\frac{1}{F_\sigma}+\frac{1}{F_\sigma}\max\{tF_\mu-(t-1),0\} \le 0,
$$
and therefore $F_{\mathbb{A}^\lor_{\sigma}(\mu^{\Box\hspace{-.47em}\lor t})}(x) =0$. Consequently, we get $F_{\mathbb{A}^\lor_{\sigma}(\mu^{\Box\hspace{-.47em}\lor t})}=F_{\mathbb{A}^\lor_\sigma(\mu)^{\Box\hspace{-.47em}\lor t}} $.

\item Consider $\sigma,\mu \in \calP(\R_{\ge 0})$. One readily observes that $x\in U_{\sigma,\mu}$ if and only if $F_\sigma(x)+F_\mu(x)-1>0$ for any $\sigma,\mu \in \calP(\R_{\ge0})$.
Since $F_{\mathbb{A}^\lor_\sigma (\mu)}(x)=F_{\mathbb{A}^\lor_\mu (\sigma)}(x)=0$ for $x\notin U_{\sigma,\mu}$ (i.e. $F_\sigma(x) + F_\mu(x)-1\le0$), we get $F_{\mathbb{A}^\lor_\sigma(\mu) \cup \hspace{-.52em} \lor \mathbb{A}^\lor_\mu(\sigma)}  (x)=0$. For $x\in U_{\sigma,\mu}$ (i.e. $F_\sigma(x) + F_\mu(x)-1>0$), we have
\begin{align*}
F_{\mathbb{A}^\lor_\sigma(\mu) \cup \hspace{-.52em} \lor \mathbb{A}^\lor_\mu(\sigma)} (x) 
&= \frac{F_{\mathbb{A}^\lor_\sigma (\mu)}(x)F_{\mathbb{A}^\lor_\mu (\sigma)}(x)}{F_{\mathbb{A}^\lor_\sigma (\mu)}(x)+F_{\mathbb{A}^\lor_\mu (\sigma)}(x)-F_{\mathbb{A}^\lor_\sigma (\mu)}(x)F_{\mathbb{A}^\lor_\mu (\sigma)}(x)}\\
&=\frac{1 -\frac{\overline{F}_\mu(x)}{F_\sigma(x)}-\frac{\overline{F}_\sigma(x)}{F_\mu(x)}+\frac{\overline{F}_\sigma(x)\overline{F}_\mu(x)}{F_\sigma(x)F_\mu(x)}}{1-\frac{\overline{F}_\sigma(x)\overline{F}_\mu(x)}{F_\sigma(x)F_\mu(x)}}\\
&=F_\sigma(x)+F_\mu(x)-1.
\end{align*}
Consequently, the desired result is obtained.
\end{enumerate}
\end{proof}

To conclude this section, we present examples of subordination functions for free max-convolution together with explicit formulas. For $\alpha>0$, let $\mu_\alpha$ denote the Dagum distribution (also called the Boolean Fr\'{e}chet distribution), defined by
\begin{align}\label{eq:Dagum}
F_{\mu_\alpha}(x) := \frac{1}{1+x^{-\alpha}} \mathbf{1}_{(0,\infty)}(x),
\end{align}
and let $\nu_\alpha$ denote the Pareto distribution (also called the free Fr\'{e}chet distribution), defined by
\begin{align}\label{eq:Pareto}
F_{\nu_\alpha}(x):= (1-x^{-\alpha})\mathbf{1}_{[1,\infty)}(x).
\end{align}

\begin{ex}\label{ex:freesub}
Let $\mu \in \calP(\R_{\ge0})$. By Theorem \ref{main2} (4), we have $\mu^{\Box\hspace{-.52em} \lor 2} = \mathbb{A}^\lor_\mu(\mu)^{\cup \hspace{-.47em}\lor 2}$. In particular, this yields the representation $\mathbb{A}^\lor_\mu(\mu) = (\mu^{\Box\hspace{-.52em} \lor 2})^{\cup \hspace{-.52em}\lor 1/2}$. 

In \cite{U22}, we introduce the max-Belinschi-Nica semigroup $\{\mathbb{B}_t^\lor \}_{t\ge0}$, defined by
$$
\mathbb{B}_t^\lor: \calP(\R_{\ge0})\to \calP(\R_{\ge0}), \quad \mu \mapsto (\mu^{\Box\hspace{-.52em} \lor (1+t)})^{\cup \hspace{-.52em}\lor \frac{1}{1+t}}, \qquad t\ge0,
$$
which provides an important connection between Boolean and free max-convolutions. It is known from \cite[Proposition 10]{U22} that $\mathbb{B}_1^\lor (\mu_\alpha)= \nu_\alpha$ for all $\alpha>0$.

Combining the above observations, we obtain 
$$
\mathbb{A}^\lor_{\mu_\alpha}(\mu_\alpha)=\mathbb{B}_1^\lor(\mu_\alpha)=\nu_\alpha, \qquad \alpha>0,
$$
and consequently,
$$
\mu_\alpha \Box\hspace{-.92em} \lor \mu_\alpha= \nu_\alpha \cup \hspace{-.90em}\lor \nu_\alpha = \mu_\alpha \lor \nu_\alpha, \qquad \alpha>0.
$$
\end{ex}

\begin{ex}
Let $\alpha,\beta>0$. A straightforward computation shows that
$$
F_{\mu_\alpha \Box\hspace{-.52em}\lor \mu_\beta}(x) = \frac{1-x^{-(\alpha+\beta)}}{(1+x^{-\alpha})(1+x^{-\beta})} \mathbf{1}_{[1,\infty)}(x).
$$
Define $\nu_{\alpha,\beta}:= \mathbb{A}_{\mu_\alpha}^\lor(\mu_\beta)$. From the preceding example, we observe that $\nu_{\alpha,\alpha}=\nu_\alpha$. By definition, we obtain
$$
F_{\nu_{\alpha,\beta}}(x) = \frac{1-x^{-(\alpha+\beta)}}{1+x^{-\beta}}\mathbf{1}_{[1,\infty)}(x).
$$
By symmetry, it follows that
$$
F_{\nu_{\beta,\alpha}}(x) = \frac{1-x^{-(\alpha+\beta)}}{1+x^{-\alpha}}\mathbf{1}_{[1,\infty)}(x).
$$
Consequently, $\mu_\alpha \Box\hspace{-.92em} \lor \mu_\beta= \nu_{\alpha,\beta} \cup \hspace{-.90em}\lor \nu_{\beta,\alpha} = \mu_\alpha \lor \nu_{\alpha,\beta}$.
\end{ex}

\section{Subordination functions for Boolean max-convolution}
\label{sec4}

In this section, we investigate the Boolean analogue of Proposition \ref{prop2} and Theorem \ref{main2}. Recall that
$$
\calQ:= \left\{(\sigma,\mu)\in \calP(\R_{\ge0})^2:  \frac{F_\sigma}{F_\mu} \text{ is monotonically decreasing on } (\alpha_\mu,\infty) \text{ and } \mu(\{\alpha_\mu\})=0\right\}.
$$
Note that, if $\sigma\in \calP(\R_{\ge0})$ satisfies $\sigma(\{\alpha_\sigma\})=0$, then $(\sigma,\sigma) \in \calQ$. Thus, $\calQ \neq \emptyset$.
\begin{lem}\label{lem:alpha}
If $(\sigma,\mu)\in \calQ$, then $\alpha_\sigma \le \alpha_\mu$.
\end{lem}
\begin{proof}
If $\alpha_\mu < \alpha_\sigma$, then there exists $x \in (\alpha_\mu, \alpha_\sigma)$ such that $F_\mu(x)>0$ and $F_\sigma(x)=0$. Then $F_\sigma(x)/F_\mu(x)=0$. But, for $y>\alpha_\sigma$, we have $F_\sigma(y)/F_\mu(y)>0$. Consequently $F_\sigma/F_\mu$ is not monotonically decreasing on $(\alpha_\mu,\infty)$, and hence $(\sigma,\mu) \notin \calQ$.
\end{proof}

\begin{prop}\label{prop:Booleansub}
Let us consider $(\sigma,\mu)\in \calQ$. Then there exists a unique measure $\mathbb{U}^\lor_\sigma(\mu) \in \calP(\R_{\ge0})$ such that
$$
\sigma \cup \hspace{-.89em}\lor \mu = \sigma \lor \mathbb{U}^\lor_\sigma(\mu).
$$
\end{prop}
\begin{proof}
For $(\sigma,\mu) \in\calQ$ and $\sigma\neq \mu$, we define
\begin{align}\label{def:Booleanmaxsub}
F:=\begin{cases}
\dfrac{F_\mu}{F_\sigma+ F_\mu - F_\sigma F_\mu}, & x\ge \alpha_\mu\\
0, & x<\alpha_\mu,
\end{cases}
\end{align}
where we always understand $F(\alpha_\mu)=0$ when $\alpha_\mu=\alpha_\sigma$.
By the definition of $F$, it is right-continuous on $\R$, $\lim_{x\to\infty} F(x)=1$ and $\lim_{x\to -\infty}F(x)=0$. For any $\alpha_\mu < x<y$, we get $F_\sigma(x)/F_\mu(x) \ge F_\sigma(y)/F_\mu(y)$, and therefore
\begin{align*}
F(x) = \frac{1}{\frac{F_\sigma(x)}{F_\mu(x)} +1-F_\sigma(x)} \le \frac{1}{\frac{F_\sigma(y)}{F_\mu(y)} +1-F_\sigma(y)} = F(y).
\end{align*}
Thus, $F$ is a distribution function of some probability measure $\mathbb{U}^\lor_\sigma(\mu)$ on $\R_{\ge0}$, and also $\sigma \cup \hspace{-.89em}\lor \mu = \sigma \lor \mathbb{U}^\lor_\sigma(\mu)$.

For $(\sigma, \sigma) \in \calQ$, we define
\begin{align}\label{def:Booleanmaxsub2}
F:= 
\begin{cases}
\dfrac{1}{2-F_\sigma}, & x\ge \alpha_\sigma\\
0, & x<\alpha_\sigma.
\end{cases}
\end{align}
One can see that $F$ is a distribution function of some probability measure $\mathbb{U}_\sigma^\lor(\sigma)$ and $\sigma \cup \hspace{-.89em}\lor \sigma = \sigma \lor \mathbb{U}^\lor_\sigma(\sigma)$.
\end{proof}

\begin{defn}
For $(\sigma,\mu) \in \calQ$, we define $\mathbb{U}_\sigma^\lor(\mu)$ as the probability measure on $\R_{\ge0}$ whose distribution function is given by \eqref{def:Booleanmaxsub} or \eqref{def:Booleanmaxsub2}. We call the distribution function $F_{\mathbb{U}^\lor_\sigma(\mu)}$ the {\it subordination function for Boolean max-convolution}. Notice that, $\alpha_{\mathbb{U}_\sigma^\lor(\mu)}=\alpha_\mu$.
\end{defn}

\begin{proof}[Proof of Theorem \ref{main3}]
\begin{enumerate}[\rm (1)]
\item We consider the case $(\sigma_1\lor \sigma_2,\mu)\in \calQ$. Assume that $F_{\sigma_1}/F_{\mu}$ is not monotonically decreasing on $(\alpha_\mu,\infty)$, then
$\frac{F_{\sigma_1}F_{\sigma_2}}{F_\mu} = \frac{F_{\sigma_1}}{F_\mu} \cdot F_{\sigma_2}$
is also not monotonically decreasing on $(\alpha_\mu,\infty)$. This contradicts the assumption that $(\sigma_1\lor \sigma_2,\mu)\in \calQ$. Hence, $F_{\sigma_1}/F_{\mu}$ must be monotonically decreasing on $(\alpha_\mu,\infty)$, and therefore $(\sigma_1,\mu)\in \calQ$. By the same argument, we also obtain $(\sigma_2,\mu)\in\calQ$. By Proposition \ref{prop:Booleansub}, the measure $\mathbb{U}_{\sigma_2}^\lor (\mu)$ is well-defined. Furthermore, since $(\sigma_1, \mathbb{U}_{\sigma_2}^\lor(\mu))\in \calQ$, the measure $(\mathbb{U}^\lor_{\sigma_1} \circ \mathbb{U}^\lor_{\sigma_2} )(\mu)$ is also well-defined. A direct computation shows that
\begin{align*}
F_{\mathbb{U}^\lor_{\sigma_1} \circ \mathbb{U}^\lor_{\sigma_2} (\mu)} 
&= F_{\mathbb{U}^\lor_{\sigma_1} (\mathbb{U}^\lor_{\sigma_2}(\mu) )}
= \frac{F_{\mathbb{U}^\lor_{\sigma_2}(\mu)}}{F_{\sigma_1}+F_{\mathbb{U}^\lor_{\sigma_2}(\mu)} - F_{\sigma_1}F_{\mathbb{U}^\lor_{\sigma_2}(\mu)} }\\
&=\frac{\frac{F_\mu}{F_{\sigma_2}+F_\mu - F_{\sigma_2}F_\mu}}{F_{\sigma_1}+\frac{F_\mu}{F_{\sigma_2}+F_\mu - F_{\sigma_2}F_\mu} - F_{\sigma_1}\cdot \frac{F_\mu}{F_{\sigma_2}+F_\mu - F_{\sigma_2}F_\mu}}\\
&= \frac{F_\mu}{F_{\sigma_1}F_{\sigma_2} + F_\mu - F_{\sigma_1}F_{\sigma_2}  F_\mu}=F_{\mathbb{U}_{\sigma_1\lor \sigma_2}^\lor (\mu)}.
\end{align*}

\item Consider $(\sigma,\mu_1), (\sigma,\mu_2)\in \calQ$. A simple computation shows that 
\begin{align*}
\frac{F_\sigma}{F_{\mu_1\cup\hspace{-.52em}\lor \mu_2}} 
= \frac{F_\sigma}{F_{\mu_1}F_{\mu_2}}(F_{\mu_1}+F_{\mu_2}- F_{\mu_1}F_{\mu_2})= \frac{F_\sigma}{F_{\mu_1}}+ \frac{F_\sigma}{F_{\mu_2}} - F_\sigma.
\end{align*}
Since $\alpha_{\mu_1\cup\hspace{-.52em}\lor \mu_2}=\max\{\alpha_{\mu_1}, \alpha_{\mu_2}\}$, the assumption implies that $F_\sigma/F_{\mu_1}$ and $F_\sigma/F_{\mu_2}$ are monotonically decreasing on $(\alpha_{\mu_1\cup\hspace{-.52em}\lor \mu_2}, \infty)$. Clearly, $(\mu_1 \cup\hspace{-.98em}\lor \mu_2) (\{\alpha_{\mu_1\cup\hspace{-.52em}\lor \mu_2}\})=0$. Consequently, $(\sigma, \mu_1\cup\hspace{-.89em}\lor \mu_2) \in \calQ$. By Proposition \ref{prop:Booleansub}, the measure $\mathbb{U}^\lor_{\sigma}(\mu_1  \cup\hspace{-.89em} \lor \mu_2)$ is well-defined. One can easily see that
\begin{align*}
F_{\mathbb{U}^\lor_{\sigma}(\mu_1  \cup\hspace{-.52em} \lor \mu_2)} = \frac{F_{\mu_1}F_{\mu_2}}{F_\sigma (F_{\mu_1}+F_{\mu_2}-2F_{\mu_1}F_{\mu_2})+F_{\mu_1}F_{\mu_2}} = F_{\mathbb{U}^\lor_{\sigma}(\mu_1)\cup\hspace{-.52em} \lor  \mathbb{U}^\lor_{\sigma}  (\mu_2)}.
\end{align*}

\item Consider $(\sigma,\mu)\in \calQ$ and $t\ge 1$. Then
\begin{align*}
\frac{F_\sigma}{F_{\mu^{\cup\hspace{-.35em}\lor t}}} = \frac{F_\sigma}{F_\mu}(t-(t-1)F_\mu)
\end{align*}
is monotonically decreasing on $(\alpha_{\mu^{\cup\hspace{-.47em}\lor t}},\infty)=(\alpha_\mu,\infty)$. Clearly, $\mu^{\cup\hspace{-.52em}\lor t}(\{\alpha_{\mu^{\cup\hspace{-.47em}\lor t}}\})=0$. By Proposition \ref{prop:Booleansub}, the measure $\mathbb{U}^\lor_{\sigma}(\mu^{\cup\hspace{-.52em}\lor t})$ is well-defined. It is easy to obtain
$$
F_{\mathbb{U}^\lor_{\sigma}(\mu^{\cup\hspace{-.47em}\lor t})}=\frac{F_\mu}{tF_\sigma \overline{F}_\mu + F_\mu}=F_{\mathbb{U}^\lor_{\sigma}(\mu)^{\cup\hspace{-.47em}\lor t}}.
$$

\item For $\sigma \in \calP(\R_{\ge0})$ with $\sigma(\{\alpha_\sigma\})=0$, we obtain
\begin{align*}
F_{\mathbb{U}^\lor_\sigma(\sigma)^{  \Box\hspace{-.47em}\lor 2}} 
=\max\left\{ \frac{F_\sigma }{2- F_\sigma },0\right\}
=\max\{ F_{\sigma^{\cup \hspace{-.47em}\lor 2}}, 0\} =F_{\sigma^{\cup \hspace{-.47em}\lor2}}.
\end{align*}
\end{enumerate}
\end{proof}

\begin{rem}
If $(\sigma,\mu),(\mu,\sigma)\in \calQ$, then $\alpha_\sigma= \alpha_\mu$ by Lemma \ref{lem:alpha}. Hence $F_\sigma/F_\mu$ is monotonically decreasing and monotonically increasing on $(\alpha_\mu,\infty)$, and therefore $F_\sigma/F_\mu$ coincides with some constant $k$ on $(\alpha_\mu,\infty)$. Since $\lim_{x\to \infty} F_\sigma(x)=\lim_{x\to \infty} F_\mu(x)=1$, it must be $k=1$. Here, we assume that $\sigma \neq \mu$. Then there exists $x\in (\alpha_\mu,\infty)$ such that $F_\sigma(x) \neq F_\mu(x)$, but this contradicts the result that $F_\sigma/F_\mu=1$ on $(\alpha_\mu,\infty)$. Consequently, we obtain $\sigma=\mu$.  Therefore, we can not establish the Boolean analogue of Theorem \ref{main2} (4) when $\sigma \neq \mu$. 
\end{rem}

To conclude this section, we give an explicit example of a subordination function for Boolean max-convolution. Recall that $\mu_\alpha$ and $\nu_\alpha$ denote the Dagum and Pareto distributions with parameter $\alpha>0$, respectively, introduced in \eqref{eq:Dagum} and \eqref{eq:Pareto}.

\begin{ex}
One can verify that $\alpha_{\mu_\alpha}=0$, $\alpha_{\nu_\alpha}=1$, and that
$$
\frac{F_{\mu_\alpha}(x)}{F_{\nu_\alpha}(x)} = \frac{1}{1-x^{-2\alpha}}
$$
is monotonically decreasing on $(1,\infty)$. Therefore $(\mu_\alpha,\nu_\alpha)\in \calQ$. By \eqref{def:Booleanmaxsub}, we obtain
$$
F_{\mathbb{U}_{\mu_\alpha}^\lor (\nu_\alpha)}(x) = \frac{1-x^{-2\alpha}}{1+x^{-\alpha}-x^{-2\alpha}}\mathbf{1}_{[1,\infty)}(x).
$$
On the other hand, $(\nu_\alpha,\mu_\alpha)\notin \calQ$ by Lemma \ref{lem:alpha}.
\end{ex}

\subsection*{Future research directions}

\begin{enumerate}[\rm (i)]
\item According to \cite{HU21} and \cite{U21}, there are interesting relationships between additive convolutions and max-convolutions in the free, Boolean, and classical settings via the law of large numbers for free multiplicative convolution. Can we find a similar relationship between monotone additive convolution and monotone (classical) max-convolution?

\item In Section \ref{sec4}, we defined the subordination functions for Boolean max-convolution. Can we define the subordination functions for Boolean additive convolution? That is, can we find the class of pairs $(\sigma,\mu)$ of probability measures on $\R$, such that there exists a unique measure $\mathbb{U}_\sigma(\mu)$ such that $\sigma \uplus \mu = \sigma \triangleright \mathbb{U}_\sigma(\mu)$?

\item Can we explain why monotone max-convolution coincides with classical max-convolution through an investigation of graph products? More precisely, we aim to construct an extreme value theory based on BMT independence, which is related to graph products (see \cite{AMVB25}), and to explain the above phenomenon within the framework of BMT independence.
\end{enumerate}

\subsection*{Acknowledgement}
The author is sincerely grateful to the referee for a careful and thorough reading of the manuscript, as well as for valuable comments and suggestions.

%%%
\bibliographystyle{amsplain}

\vspace{8mm}

\hspace{-6mm}{\bf Yuki Ueda}\\
Faculty of Education, Department of Mathematics, Hokkaido University of Education, 9 Hokumon-cho, Asahikawa, Hokkaido 070-8621, Japan\\
E-mail: ueda.yuki@a.hokkyodai.ac.jp

\end{document}